\documentclass[12pt,twoside, english]{amsart}

\usepackage{hyperref}
\usepackage{amsthm,thmtools,xcolor}

\usepackage{verbatim}
\usepackage{amsmath}
\usepackage{bm}
\usepackage{a4wide}

\usepackage{babel,latexsym,amsmath,amsthm,eucal}
\usepackage[utf8]{inputenc}

\usepackage[T1]{fontenc}
\usepackage{times}
\usepackage{amssymb,latexsym}

\usepackage{amssymb,latexsym}
\usepackage{enumerate}

\newcommand{\dbar}{\ensuremath{\overline\partial}}

\makeatletter
\newcommand{\sumprime}{\if@display\sideset{}{'}\sum%
            \else\sum'\fi}
\makeatother

\begin{document}

\numberwithin{equation}{section}

\newtheorem{theorem}{Theorem}[section]
\newtheorem{proposition}[theorem]{Proposition}
\newtheorem{conjecture}[theorem]{Conjecture}
\def\theconjecture{\unskip}
\newtheorem{corollary}[theorem]{Corollary}
\newtheorem{lemma}[theorem]{Lemma}
\newtheorem{observation}[theorem]{Observation}
\newtheorem{definition}{Definition}
\numberwithin{definition}{section} 
\newtheorem{remark}{Remark}
\def\theremark{\unskip}
\newtheorem{kl}{Key Lemma}
\def\thekl{\unskip}
\newtheorem{question}{Question}
\def\thequestion{\unskip}
\newtheorem{example}{Example}
\def\theexample{\unskip}
\newtheorem{problem}{Problem}
\newtheorem*{AsA}{\textcolor{red}{Assumption A}}

\address{DEPARTMENT OF MATHEMATICAL SCIENCES, NORWEGIAN UNIVERSITY OF SCIENCE AND TECHNOLOGY, NO-7491 TRONDHEIM, NORWAY}
\email{xu.wang@ntnu.no}

\title[An explicit estimate of the Bergman kernel]{An explicit estimate of the Bergman kernel for positive line bundles}

 \author{Xu Wang}
\date{\today}

\maketitle

{\centering\footnotesize \emph{Dedicated to Bo Berndtsson on occasion of his 70th birthday}\par}

\begin{abstract} We shall give an explicit estimate of the lower bound of the Bergman kernel associated to a positive line bundle. In the compact Riemann surface case, our result can be seen as an explicit version of Tian's partial $C^0$-estimate.
\end{abstract}



\tableofcontents

\section{Introduction}

Let $(L, e^{-\phi})$ be a positive line bundle over an $n$-dimensional complex manifold $X$. Let $m$ be a positive integer. Let $K_X$ be the canonical line bundle over $X$. We call
\begin{equation}\label{eq:K}
{\rm K}_{m\phi}(x):=\sup_{u\in H^0(X, K_X+mL)}  \frac{u(x) \wedge \overline{u(x)}\, e^{-m\phi(x)}}{\int_X   u \wedge \bar u\, e^{-m\phi} },
\end{equation}
the \emph{Bergman kernel forms} and
\begin{equation}\label{eq:B}
{\rm B}_{m\phi}(x) := \sup_{u\in H^0(X, mL)}  \frac{ |u(x)|^2 e^{-m\phi(x)} }{\int_X  |u|^2 e^{-m\phi} \,{\rm MA}_{m\phi} }, \ \ \ \ {\rm MA}_{m\phi}:=\frac{(i\partial\dbar (m\phi))^n}{n!},
\end{equation}
the \emph{Bergman kernel functions}. In \cite{T1} Tian proved that if $X$ is compact then
\begin{equation}\label{eq:T1}
\lim_{m\to \infty} \frac{{\rm K}_{m\phi}}{ {\rm MA}_{m\phi}}= \lim_{m\to \infty} {\rm B}_{m\phi}=\frac{1}{(2\pi)^n}.
\end{equation}
Effective lower bound estimate (with Ricci curvature, diameter and volume assumptions) for ${\rm B}_{m\phi}$ is known as Tian's \emph{partial $C^0$-estimate} \cite{T2}. The first general result is obtained by Donaldson--Sun \cite{DS} using proof by contradiction. Our main results are the followings:

\medskip

\noindent
\textbf{Theorem A.} \emph{Let $(L, e^{-\phi})$ be a positive line bundle over a compact Riemann surface $X$. Put $\omega:={\rm MA}_{\phi}= i\partial\dbar \phi$. Denote by $ {\rm Ric}\,\omega:=i\dbar\partial \log \omega$ the Ricci form of $\omega$. Assume that 
$$ {\rm Ric}\,\omega\leq \omega,  \ \ \ L_0 \geq 2\pi,
$$
where $L_0$ denotes the infimum of the length of closed geodesics in $X$, 
then ${\rm K}_\phi / {\rm MA}_{\phi} \geq \frac{1}{8\pi}$}.

\medskip

\noindent
\textbf{Theorem B.} \emph{Let $(L, e^{-\phi})$ be a positive line bundle over a compact Riemann surface $X$. If 
$$ -\omega/2 \leq {\rm Ric}\,\omega\leq \omega /2,  \ \ \ L_0 \geq  2\pi \, \sqrt 2,
$$
then ${\rm B}_\phi \geq \frac{1}{16\pi}$}.

\medskip

\noindent
\textbf{Remark.} \emph{In case $X=\mathbb P^1$ and $\omega=2\cdot i\partial\dbar \log(1+|z|^2)$ we have
$$
 {\rm Ric}\,\omega =  \omega,  \ \ \ L_0 = 2\pi,
$$
a direct computation gives $L=-K_X$ and ${\rm K}_\phi / {\rm MA}_{\phi} = \frac{1}{4\pi}$. We do not know whether $$
{\rm K}_\phi / {\rm MA}_{\phi} \geq \frac{1}{4\pi}
$$
is always true with the assumptions in Theorem A. On the other hand, Theorem A implies $K_{m\phi}/ {\rm MA}_{m\phi} \geq 1/(8\pi)$ for every positive integer $m$. This is also near optimal since by \eqref{eq:T1}
$$\lim_{m\to \infty} K_{m\phi}/ {\rm MA}_{m\phi} =1/(2\pi).
$$
In case $ {\rm Ric}\,\omega \leq 0$, $L_0/2$ is equal to the injectivity radius. For example if $X=\mathbb C/\Gamma$ is a torus and $\omega=i\partial\dbar (|z|^2/2)$ then ${\rm Ric}\,\omega = 0$ and $
L_0=\inf_{0\neq \gamma \in \Gamma} |\gamma|
$.
}

\medskip

In the first version of this paper, a weaker version of the above theorems is proved using an Ohsawa--Takegoshi type theorem, a variant of the Blocki--Zwonek estimate \cite{BZ} and the isoperimetric inequality. Later we find that one may use the Hessian comparison theorem to simplify the proof and generalize the above theorems to the followings higher dimensional cases.  

\medskip

\noindent
\textbf{Theorem An.} \emph{Let $(L, e^{-\phi})$ be a positive line bundle over an $n$-dimensional compact complex manifold  $X$.  Assume that the sectional curvature of $\omega:=i\partial\dbar \phi$ is bounded above by $1/(4n)$ and $L_0 \geq 2\pi \, \sqrt n$ then $
{\rm K}_{\phi}/  {\rm MA}_{\phi} \geq \frac{1}{2} \,\frac{1}{(4\pi n)^n}$.}

\medskip

\noindent
\textbf{Theorem Bn.}  \emph{Let $(L, e^{-\phi})$ be a positive line bundle over an $n$-dimensional compact manifold  $X$.  Assume that the sectional curvature of $\omega:=i\partial\dbar \phi$ is bounded above by $1/(8n)$, $L_0 \geq 2\pi \, \sqrt{2n}$ and ${\rm Ric} \, \omega \geq -\omega/2$. Then $
{\rm B}_{\phi}\geq \frac{1}{2} \,\frac{1}{(8\pi n)^n}.$}

\medskip

\noindent
\textbf{Remark.} \emph{Since the Ricci curvature is certain sums of sectional curvatures, the curvature assumptions in Theorem Bn also imply a lower bound of the sectional curvature. Hence one may use \cite[Corollary 2.3.2]{HK} to find a lower bound of $L_0$ in terms of the lower bound of the volume and the upper bound of the diameter. Thus, except the upper bound of the sectional curvature, the assumptions in Theorem Bn follow from the standard assumptions in Tian's partial $C^0$-estimate (for results on Tian's partial $C^0$-estimate, see \cite{Bamler, CW1, CW2, J, JWZ, LS, S, T3, WZ, Z}, etc). Our  main contribution is the explicit constant in the estimate. Moreover, our estimate implies that ${\rm B}_{m\phi}\geq \frac{1}{2} \,\frac{1}{(8\pi n)^n}$ for all positive integers $m$. If we understand \cite{DS} correctly, for general positive line bundles, this estimate can not be true with just Ricci curvature, diameter and volume assumptions (see the explanation at the end of \cite{DS}).}

\bigskip

\emph{Acknowledgment}:  It is a pleasure to thank Bo Berndtsson for
stimulating discussions related to the topics of this article. 

\section{Hessian comparison theorem}

\begin{definition}\label{de:bs} Let $X$ be a Riemann manifold.  Denote by $K(V, W)$ the sectional curvature of the tangent plane spanned by $V, W$. Fix $x\in X$, the injectivity radius at $x \in X$ is defined as
$$
{\rm inj}(x):=\sup\{r>0: {\rm exp}_x|_{B(0,r)} \ \text{is diffeomorphism}\}
$$
where ${\rm exp}_x: T_xX \to X$ denotes the exponential map at $x$ and $B(0,r)$ denotes the ball of radius $r$ around $0\in T_xX$. We call $ {\rm inj}_X:=\inf_{x\in X} {\rm inj}(x)$ the  injectivity radius of $X$.
\end{definition}

The following $\partial\dbar$-comparison theorem is a direct consequence of the Hessian comparison theorem (see \cite[Lemma 1.13 in page 14 and Theorem A in page 19]{GWbook}). 

\begin{theorem}\label{th:comparison} Let $X_1, X_2$ be K\"ahler manifolds. Let $\gamma_1:[0,b]\to X_1$ and $ \gamma_2: [0,b]\to X_2$ be unit speed geodesics. With the definition above, suppose that
\begin{equation}\label{eq:comparison1} 
b \leq \min\{ {\rm inj}(\gamma_1(0)), {\rm inj}(\gamma_2(0))\}
\end{equation}
and for all $t \in [0,b]$, $v_1 \bot \dot \gamma_1(t)$ and $v_2 \bot \dot \gamma_2(t)$,
\begin{equation}\label{eq:comparison2} 
K(\dot \gamma_1(t), v_1) \leq K(\dot \gamma_2(t), v_2).
\end{equation}
Let $d_j:=d(\cdot, \gamma_j(0))$ be distance functions. If $f:(0, b) \to \mathbb R$ is smooth and increasing then
\begin{equation}\label{eq:comparison3} 
i\partial\dbar (f\circ d_1)(V_1, V_1) \geq i\partial\dbar (f\circ d_2)(V_2, V_2)
\end{equation}
for all $t\in (0,b)$, $V_j \in T_{\gamma_j(t)} X_j$, $j=1,2$, such that $|V_1|=|V_2|$ and 
$$
(\dot \gamma_1(t), V_1)= (\dot \gamma_2(t), V_2), \ \ \ (\dot \gamma_1(t), J V_1)= (\dot \gamma_2(t), J V_2).
$$
\end{theorem}

We shall apply the above theorem to  $X_2=\mathbb P^n$ with the Fubini study metric form  $\omega_2= 2\, i \partial\dbar \log(1+|z|^2)$. A direct computation gives 
\begin{equation}\label{eq:FS}
K(V, W) = \frac 14 \left(1+3 (V, JW)^2\right), \ \forall \ V\bot W,  \ \  {\rm inj}(x) = \pi.
\end{equation}
In particular, $K(V, W) =1$ in case $n=1$ and $K(V, W) \geq 1/4$ in case $n\geq 2$. We also need the following distance function formula on $\mathbb C^n \subset \mathbb P^n$
$$
d_2:=d(0, z) = 2 \int_0^{|z|} \frac{dx}{1+x^2}= 2\arctan |z|.
$$
Put
$$
\psi= \log \sin^2(d_2/2)
$$
we have
$$
\psi= \log \frac{\tan^2(d_2/2)}{1+\tan^2(d_2/2)} = \log |z|^2 -\log (1+|z|^2) \leq 0
$$
and
$$
i\partial\dbar \psi \geq -i\partial\dbar \log (1+|z|^2) =-\omega_2/2.
$$
Apply the above theorem to $f(x)=\log \sin^2(x/2)$ and $b=\pi$ we get:

\begin{corollary}\label{co:com} Let $(X_1, \omega_1)$ be an $n$-dimension K\"ahler manifold. Fix $x\in X_1$. Assume that the injectivity radius at $x$ is no less than $\pi$. Put
$$
\psi(z):= 
\begin{cases}
\log \sin^2(d(z,x)/2) &  d(z, x) \leq \pi \\
0 & d(z, x) > \pi. 
\end{cases}
$$
If
$$
\begin{cases}
\text{the sectional curvature of $\omega_1$ is no bigger than $1$,} &  n=1 \\
\text{the sectional curvature of $\omega_1$ is no bigger than $1/4$,} & n\geq 2
\end{cases}
$$
then $i\partial\dbar \psi \geq -\omega_1/2$ on $X$.
\end{corollary}
\begin{proof} Apply the theorem above, we get that $i\partial\dbar \psi \geq -\omega_1/2$ when $d(z, x) < \pi$. Hence the corollary follows since the gradient of $\psi$ vanishes at $d(z, x) = \pi$.
\end{proof}

\section{An Ohsawa--Takegoshi type theorem}  We shall use the following Ohsawa--Takegoshi type theorem \cite{NW}, which is a special case of the main theorem in \cite{GZ}.

\begin{theorem}\label{th:BL}  Let $(L, e^{-\phi})$ be a positive line bundle on an $n$-dimensional compact complex manifold. Fix $x\in X$. Assume that there is a non-positive function $G$ smooth outside $x$ such that
$
G(z)- \log|z-x|^{2n} $
is smooth near $x$ and 
$$
i\partial\dbar \phi  + \lambda \,  i\partial\dbar G  \geq 0
$$
on $X$ for some constant $\lambda >1$. Then ${\rm K}_{\phi}$ in \eqref{eq:K} satisfies
\begin{equation}\label{eq:BL}
{\rm K}_{\phi}(x) \geq  \frac{\lambda-1}{\lambda} \lim_{t\to-\infty}\frac{{\rm MA}_{\phi}(x)}{e^{-t}\int_{G<t} {\rm MA}_{\phi}},
\end{equation}
where ${\rm MA}_{\phi}$ is defined in \eqref{eq:B}.
\end{theorem}

\begin{proof} Let us rephrase the proof in \cite{NW}. Our curvature assumption implies that
$$
\phi^{t+is}:=\phi+\lambda \max\{G-t, 0\}
$$
defines a singular metric on $\mathbb C_{t+is} \times L$ with non-negative curvature current. Hence Berndtsson's theorem \cite{Bern09} implies that $\log {\rm K}_{\phi^t}(x)$ is a convex function of $t$. By a direct computation (see the appendix in \cite{NW} or Theorem 3.8 in \cite{BL}) we find that
$$
\lim_{t\to -\infty} e^t {\rm K}_{\phi^t}(x) =  \frac{\lambda-1}{\lambda}  \lim_{t\to-\infty}\frac{{\rm MA}_{\phi}(x)}{e^{-t}\int_{G<t} {\rm MA}_{\phi}}
$$
is finite since $G(z)-\log|z-x|^{2n}$ is smooth near $x$. Hence $e^t {\rm K}_{\phi^t}(x) =e^{t+\log {\rm K}_{\phi^t}(x)}$, as a convex function of $t$ bounded near $-\infty$, must be increasing. Thus
$$
{\rm K}_{\phi}(x) = e^0 {\rm K}_{\phi^0}(x) \geq \lim_{t\to -\infty} e^t {\rm K}_{\phi^t}(x)  =  \frac{\lambda-1}{\lambda}  \lim_{t\to-\infty}\frac{{\rm MA}_{\phi}(x)}{e^{-t}\int_{G<t} {\rm MA}_{\phi}}
$$
gives our estimate.
\end{proof}

\subsection{Proof of Theorem A, B, An, Bn} 

\begin{proof}[Proof of Theorem A, An] By Klingenberg's estimate (see \cite[Corollary 1.2]{Zw}), if the sectional curvature is no bigger than $1/c$ then we have
\begin{equation}\label{eq:inj-k}
\min\{L_0/2,  \sqrt{c}\, \pi \}   \leq {\rm inj}_X \leq L_0/2.
\end{equation}
Hence our assumptions implies that the injectivity radius of $(X, \omega/n)$, $\omega:=i\partial\dbar\phi$, is no less than $\pi$. Thus one may apply  Corollary \ref{co:com} to  $
(X_1, \omega_1)= (X, \omega/n).
$
Put $G=n\psi$. Corollary \ref{co:com} implies that
$$
i\partial\dbar\phi +2\, i\partial\dbar G \geq 0.
$$
Hence Theorem 3.1 ($\lambda=2$) gives Theorem A, An. 
\end{proof}

\begin{proof}[Proof of Theorem B, Bn] By the Ricci curvature assumption, $\phi$ and $i\partial\dbar \phi$ defines a metric on $L-K_X$ with curvature
$$
i\Theta=\omega +{\rm Ric} \,\omega \geq \omega/2.
$$
Thus one may apply  Corollary \ref{co:com} to  $
(X_1, \omega_1)= (X, \omega/(2n)).$ Put $G=n\psi$ then 
$$
i\Theta + 2\, i\partial\dbar G \geq 0.
$$
Apply Theorem 3.1 to $L-K_X$, we get Theorem B, Bn. 
\end{proof}

\section{Another proof of a weaker version of Theorem A, B}

In this section we shall give another proof of Theorem A, B with an extra volume assumption. 

\medskip

\noindent
\textbf{Theorem C.} \emph{Let $(L, e^{-\phi})$ be a positive line bundle over a compact Riemann surface $X$. If
$$ \int_X \omega \geq 8\pi, \ \ \ {\rm Ric}\,\omega\leq \omega,  \ \ \ L_0 \geq 2\pi,
$$
then ${\rm K}_\phi / {\rm MA}_{\phi} \geq \frac{1}{8\pi}$}.

\medskip

\noindent
\textbf{Theorem D.} \emph{Let $(L, e^{-\phi})$ be a positive line bundle over a compact Riemann surface $X$. If 
$$  \int_X \omega \geq 16\pi , \ \ -\omega/2 \leq {\rm Ric}\,\omega\leq \omega /2,  \ \ \ L_0 \geq  2\pi \, \sqrt 2,
$$
then ${\rm B}_\phi \geq \frac{1}{16\pi}$}.

\subsection{The Blocki--Zwonek estimate}

We shall study the right hand side of  \eqref{eq:BL} using a variant of Blocki--Zwonek's estimate \cite[Proof of Theorem 3]{BZ} (see also \cite[section 10]{Blocki2} for related results).

\begin{lemma}\label{le:BZ} With the notation in Theorem \ref{th:BL}. Let $\omega$ be an arbitrary K\"ahler form on  $X$. Then 
$$
\frac{d}{dt} \int_{G<t} \omega_n \geq  \frac{\sigma(G=t)^2}{2\int_{G<t} i\partial\dbar G\wedge \omega_{n-1}}, \ \ \ \  \omega_{q}:=\omega^{q}/q!, 
$$
where 
$$
\sigma(G=t):=\int_{G=t} d\sigma , \ \ \ d\sigma:=\sqrt 2\,\sum \frac{G_{\bar \alpha} \omega^{\bar \alpha \beta}}{|\dbar G|_{\omega}} \frac{\partial}{\partial z^\beta}  \, \rfloor \, \omega_n,
$$
is the measure of the hypersurface $\{G=t\}$ with respect to $\omega$
\end{lemma}

\begin{proof} Notice that
$$
V:=\frac{\partial}{\partial t} + \sum \frac{G_{\bar \alpha} \omega^{\bar \alpha \beta}}{|\dbar G|^2_{\omega}} \frac{\partial}{\partial z^\beta}
$$
satisfies $V(G-t)=0$, hence it can be used to compute $\frac{d}{dt} \int_{G<t}$, in particular, we have
$$
\frac{d}{dt} \int_{G<t} \omega_n = \int_{G<t} L_{V} \omega_n= \int_{G=t}  V\, \rfloor \, \omega_n = \frac{1}{\sqrt2} \int_{G=t}  \frac{d\sigma}{|\dbar G|_{\omega}}.
$$
Hence the Cauchy--Schwarz inequality gives
$$
\frac{d}{dt} \int_{G<t} \omega_n  \geq \frac{\sigma(G=t)^2}{\sqrt 2\int_{G=t}  |\dbar G|_{\omega} \, d\sigma} = \frac{\sigma(G=t)^2}{2\int_{G=t}  i\dbar G \wedge \omega_{n-1}}=  \frac{\sigma(G=t)^2}{2\int_{G<t}  i\partial\dbar G \wedge \omega_{n-1}},
$$
where we use the Stokes theorem in the last equality.
\end{proof}

Since
$$
G_1 \leq G_2 \Rightarrow  \int_{G_2<t} \omega_n  \leq  \int_{G_1<t} \omega_n, 
$$
in order to get the best estimate from \eqref{eq:BL}, one should choose $G$ to be the following \emph{envelope}, say $g_{\phi, x, \lambda}$, defined by
\begin{equation}\label{eq:envelope}
\sup \{G\leq 0: \textcolor{blue}{\text{$
G(z)- \log|z-x|^{2n} $
is smooth near $x$ and 
$
i\partial\dbar \phi  + \lambda \,  i\partial\dbar G  \geq 0
$}}\}.
\end{equation}
It is known that (see \cite{Dem-sin})
\begin{equation}\label{eq:seshadri}
\epsilon_{x}(L):=\sup\{\lambda\geq 0:  \text{there exists $G\leq 0$ on $X$ such that the \textcolor{blue}{blue} part in \eqref{eq:envelope} holds}\}
\end{equation}
is equal to the \emph{Seshadri constant} up to a constant factor $n$. If $0<\lambda <\epsilon_{x}(L)$ then $g_{\phi, x, \lambda}$, as an envelope, must satisfy
\begin{equation}\label{eq:envelope1}
(i\partial\dbar \phi  + \lambda \,  i\partial\dbar g_{\phi, x, \lambda})^n =  (2\pi n \lambda)^n \delta_x
\end{equation}
on $\{g_{\phi, x, \lambda} <0\}$, where $\delta_x$ is the Dirac measure defined by $\int_{X}f \, \delta_x =f(x)$.
Thus if we choose $G=g_{\phi, x, \lambda}$ and $\omega =i\partial\dbar \phi$ then
\begin{align}
 \int_{G<t}  i\partial\dbar G \wedge \omega_{n-1} & = \int_{G<t}  (i\partial\dbar G + \omega /\lambda - \omega /\lambda) \wedge \omega_{n-1} \\
\label{eq:b-n1} & = \int_{G<t}  (i\partial\dbar G + \omega /\lambda) \wedge \omega_{n-1} -\frac{n}{\lambda} \int_{G<t}\omega_{n} \\
\label{eq:envelope2} & \leq \int_{G<0}  (i\partial\dbar G + \omega /\lambda) \wedge \omega_{n-1} -\frac{n}{\lambda} \int_{G<t}\omega_{n}  \\
\label{eq:b-n2} & =  \frac{n}{\lambda} \left( \int_{G<0}\omega_{n} - \int_{G<t}\omega_{n} \right),
\end{align}
where we use $\int_{G<0}  i\partial\dbar G \wedge \omega_{n-1} =0$ (since $G=0$ outside $\{G<0\}$) in \eqref{eq:b-n2}. In case $n=1$, \eqref{eq:envelope1} directly gives
\begin{equation}\label{eq:envelope3}
\int_{G<t}  i\partial\dbar G = \int_{G<t}  i\partial\dbar G + \omega /\lambda - \omega /\lambda = 2\pi - \int_{G<t}  \omega /\lambda .
\end{equation}
Hence Lemma \ref{le:BZ} implies
\begin{equation}\label{eq:envelope4}
\frac{d}{dt} \int_{G<t} \omega \geq  \frac{\sigma(G=t)^2}{4\pi - \frac2{\lambda} \int_{G<t} \omega }.
\end{equation}

\subsection{Isoperimetric inequality}

We shall use the following result (see inequality (5.4) in \cite[Proposition 5.2]{MJ}).

\begin{lemma}[Isoperimetric inequality]\label{le:iso}  Let $U$ be an open subset of a compact Riemann surface $(X, \omega)$. Assume that
$$
A:=\int_U \omega \leq \frac12  \int_X \omega, \ \ {\rm Ric}\, \omega \leq k\, \omega.
$$ 
Then 
\begin{equation}\label{eq:iso}
\sigma(\partial U)^2 \geq \min\{L_0^2, A(4\pi -k A)\},
\end{equation}
where $L_0$ denotes the infimum of the length of simple closed geodesics in $X$.
\end{lemma}

\begin{proof} By the definition of the Seshadri constant in \eqref{eq:seshadri}, we know that in case $n=1$,  the Hodge decomposition gives
\begin{equation}\label{eq:seshadri1}
\epsilon_x = {\rm deg} (L):=\int_X c_1(L) = \frac{1}{2\pi}\int_X \omega.
\end{equation}
Hence if $\int_X \omega \geq 8
\pi$ then $\epsilon_x \geq 4$. Hence we can take
$\lambda=2$ in \eqref{eq:envelope4}, which gives
\begin{equation}\label{eq:envelope5}
\frac{d}{dt} \int_{G<t} \omega \geq  \frac{\sigma(G=t)^2}{4\pi - \int_{G<t} \omega }.
\end{equation}
Now, since $\ {\rm Ric}\,\omega \leq \omega$, $L_0^2 \geq 4\pi^2$ and
$$
\int_{G<t} \omega \leq \int_{G<0} \omega =2\pi \lambda =4\pi \leq \frac12 \int_X \omega,
$$
one may apply the above lemma to $U=\{G<t\}$, then by \eqref{eq:iso} we have
$$
\sigma(G=t)^2 \geq \min\{4\pi^2, A (4\pi-A)\}= A(4\pi -A),   \ \ \ A:=\int_{G<t} \omega.
$$
Thus \eqref{eq:envelope5} gives
$$
\frac{d}{dt} \int_{G<t} \omega \geq  \int_{G<t} \omega,
$$
which implies that $e^{-t} \int_{G<t} \omega$ is increasing with respect to $t<0$. Hence
$$
\lim_{t\to -\infty} e^{-t} \int_{G<t} \omega \leq e^{-0} \int_{G<0} \omega = 2\pi \lambda =4\pi.
$$
Then by  \eqref{eq:BL} we have $K_{\phi} \geq \frac{\omega}{8\pi}$. 
\end{proof}

Theorem D follows by a similar argument (see the difference between the proof of Theorem A, B above).

\end{document}